\documentclass[a4paper]{article}

\usepackage{amsthm}
\usepackage{amsmath}
\usepackage{amsfonts}
\usepackage{amssymb}

\usepackage{bbm}

\usepackage{csquotes}

\usepackage{color}

\usepackage{tikz}

\usepackage{enumitem}

\usepackage{url}

\usepackage{hyperref}
\usepackage{cleveref}

\usepackage{titling}

\newtheorem{theorem}{Theorem}
\newtheorem{proposition}{Proposition}

\newtheorem{lemma}{Lemma}

\theoremstyle{definition}
\newtheorem{definition}{Definition}

\newcommand{\gen}[1]{\langle #1 \rangle_K}

\newcommand{\comment}[1]{}

\newcommand{\card}[1]{|#1|}

\newcommand{\myset}[1]{\{1,\dots,#1\}}

\newcommand{\id}{\mathbbm{1}}

\newcommand{\summand}[2]{\frac{1}{#1} \id_{#2}}
\newcommand{\term}[1]{\summand{\card{#1}}{#1}}

\usepackage[affil-it]{authblk}
\title{Minimal nontrivial solutions of the isometry equation}
\author{Serhii Dyshko
	\thanks{Electronic address: \texttt{dyshko@univ-tln.fr}}}
\affil{Institut de math\'ematiques de Toulon, Universit\'e de Toulon, France}

\begin{document}
\maketitle

\begin{abstract}
	In the paper there are described minimal nontrivial solutions of the isometry equation. This equation naturally appears in the coding theory in the study of additive code isometries. The nontrivial minimal solutions correspond to the case of unextendible additive isometries of the shortest code length. Based on this full description, several useful properties of minimal nontrivial solutions were observed.
\end{abstract}

\section{Preliminaries}
\newcommand{\V}{\mathcal{V}}
\newcommand{\U}{\mathcal{U}}
 
Let $K$ be a finite field of size $\card{K}= q$ and let $W$ be a vector space over $K$ of dimension greater than one. Let $m$ be a positive integer and let $V_i, U_i \subseteq W$ be vector spaces, where $i \in \myset{m}$.
 		
		Recall for the pair of sets $X \subseteq Y$ the indicator function $\id_X : Y \rightarrow \{0,1\}$ is defined as $\id_X(x) = 1$ for $x \in X$ and $\id_X(x) =0$ otherwise.
		
The following equation is called an \emph{isometry equation},
		\begin{equation}\label{eq-main-counting-space}
			\sum_{i=1}^m \frac{1}{\card{V_i}} \id_{V_i}=
			\sum_{i=1}^m \frac{1}{\card{U_i}} \id_{U_i}\;.
		\end{equation}
		
Denote $\V = (V_1, \dots, V_m)$, $\U = (U_1,\dots,U_m)$ and call the \emph{tuples of spaces}.
A pair of tuples $(\U, \V)$ is called a \emph{solution} if it satisfies \cref{eq-main-counting-space}.
		
The easiest way to find a solution is to chose any spaces $V_1, \dots, V_m \subseteq W$ and define $U_i = V_{\pi(i)}$, for some permutation $\pi \in S_m$, where $i \in \myset{m}$.
We say that tuples $\V$ and $\U$ are \emph{equivalent} ($\U \sim \V$) if there exists a permutation $\pi \in S_m$, such that $V_i  = U_{\pi(i)}$, for all $i \in \myset{m}$.
Such a solution $(\U,\V)$, where the tuples $\U$ and $\V$ are equivalent, is called \emph{trivial}. Note that the defined equivalence of tuples is really an equivalence relation.

We say that two pairs $(\U,\V)$ and $(\U',\V')$ are equivalent (denote $(\U,\V) \sim (\U',\V')$) if $\U \sim \U'$, $\V \sim \V'$ or $\V \sim \U'$, $\U \sim \V'$. The defined equivalence of pairs is also an equivalence relation on the set of all pairs of tuples of spaces. A pair $(\U,\V)$ is a solution if and only if any equivalent is a solution. Moreover, $(\U,\V)$ is a trivial solution if and only if any equivalent pair is a trivial solution. 
		
In general, not all the solutions are trivial. Denote by $\mathbb{P}_1(K)$ a projective space of dimension one over $K$. Note that $\card{\mathbb{P}_1(K)} = \card{K} + 1 = q + 1$. For $m = q + 1$ there exists an example of a nontrivial solution.
\begin{definition}
	A pair $(\U,\V)$ is called a pair of \emph{Type A}, if there exist a subspace $S \subseteq W$ of dimension $k$ and two different vectors $a,b \in W$, with $S \cap \gen{a,b} = \{0\}$, such that $V_1 = \dots = V_{q} = \gen{S,a,b}$, $V_{q+1} = S$ and $U_i = \gen{S, \alpha_i a + \beta_i b}$, for $i \in \myset{q+1}$, where $[\alpha_i: \beta_i]$ is the $i$th element of $\mathbb{P}_1(K)$.
		\end{definition}
In fact, the spaces $U_1,\dots, U_m$ from a pair of Type A are all different hyperplanes in $\gen{S,a,b}$ that contain the subspace $S$. In \cite{d1} it was proved that a pair of Type A is a nontrivial solution. Indeed, denote $V = \gen{S, a, b}$, then
		
\begin{equation*}
\sum_{i=1}^{q+1} \term{V_i} = q \summand{q^{k+2}}{V} + \summand{q^{k}}{S} = \frac{1}{q^{k+1}} \left( \id_V + q \id_S \right) \;,
\end{equation*}
		
\begin{equation*}
\sum_{i=1}^{q+1} \term{U_i} = \frac{1}{q^{k+1}}\sum_{i=1}^{q+1} \id_{U_i} = \frac{1}{q^{k+1}} \left( \id_V + q \id_S \right)\;.
\end{equation*}
Evidently, a solution of Type A is nontrivial. The inclusion diagram of spaces from a pair of Type A is presented in \Cref{fig-type-a}.
	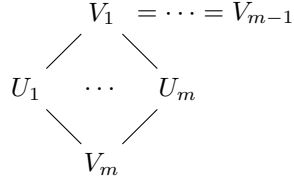
\begin{figure}[!ht]
		\centering
		\begin{tikzpicture}
		\node (top) at (0,2) {$V_1$};
		\node (topr) at (1.5,2) {$ = \dots = V_{m-1}$};
		\node (l) at (-1, 1) {$U_1$};
		\node (c) at (0, 1) {$\dots$};
		\node (r) at (1, 1) {$U_m$};
		\node (down) at (0, 0) {$V_m$};
		\draw (top) -- (l) -- (down) -- (r) -- (top);
		\end{tikzpicture}
		\caption{Solution of Type A}
		\label{fig-type-a}
	\end{figure}

\comment{In this paper we are concerned with the nontrivial solutions. The importance of study of the nontrivial solutions is because it is related to the description of code isometries. In \cite{d1} it was proved that there exists a nontrivial solutions if and only if there exists an unextendible additive isometry of an additive code. There the were also proven some properties of nontrivial solutions and minimum requirement for their existence.
}

To classify all the solutions, up to equivalence, for some $m$, we have to describe all trivial and all nontrivial solutions. For trivial solutions the task is easy --- all such solutions are parametrized by tuples of spaces of length $m$, where the spaces are subspaces of $W$.

The case of nontrivial solutions is more complicated. We introduce several important properties of nontrivial solutions that we are using further.

\begin{lemma}\label{lemma-minimum-covering-number-of-one-space}
	Let $V$ be a nonzero vector space over $K$ and let $U_i \subset V$ be proper subspaces, for $i \in \myset{m}$. If $V = \bigcup_{i =1}^m U_i$, then $m$ is greater than the cardinality of ${K}$.
\end{lemma}
\begin{proof}
	For any $i \in \myset{m}$, $\dim_K U_i \leq \dim_K V - 1$ and hence $\card{U_i} \leq \frac{\card{V}}{\card{K}}$. Thus we have
	\begin{equation*}
		\card{V} < \sum_{i =1}^m \card{U_i} \leq m \frac{\card{V}}{\card{K}}
	\end{equation*} that implies $m > \card{K}$.
\end{proof}

\begin{lemma}\label{lemma-minimum-size-of-space-equation}
	Let $U_1,\dots, U_r, V_1, \dots, V_s$ be different vector spaces over $K$. Assume that $a_1,\dots, a_r, b_1,\dots, b_s > 0$ and
	\begin{equation*}
	\sum_{i = 1}^{r} a_i \id_{U_i} = \sum_{i = 1}^{s} b_i \id_{V_i} \,\, .
	\end{equation*}
	Then $\max\{r,s\}$ is greater than the cardinality of ${K}$.
\end{lemma}
\begin{proof}
	Among the spaces $V_1, \dots, V_s, U_1, \dots, U_r$ choose one that is maximal under inclusion. It is either $V_i$ for some $i \in \myset{s}$, or $U_j$ for some $j \in \myset{t}$. In the first case $V_i = \bigcup_{j = 1}^{r} (V_i \cap U_j)$, where for all $j \in \myset{r}$, $V_i \cap U_j \subset V_i$. From Lemma \ref{lemma-minimum-covering-number-of-one-space}, $r > \card{K}$. Similarly, in the second case $s > \card{K}$.
\end{proof}

\begin{proposition}\label{thm-nontrivial-solution-minimal-length}
	There exists a nontrivial solution if and only if $m \geq q +1$.
\end{proposition}
\begin{proof}
	Let $(\U,\V)$ be a nontrivial solution. Simplify \cref{eq-main-counting-space} by combining and elimination of all equal spaces. The resulting equation is in the form of the equation from \Cref{lemma-minimum-size-of-space-equation} and therefore $m > q$.
	
	Conversely, let $(\U,\V)$ be of Type A. If $m = q+ 1$ we have already shoved that $(\U,\V)$ is a nontrivial solution. If $m > q + 1$, let $X_1,\dots, X_{m-q-1}$ be subspaces in $W$. Define tuples $\U',\V'$ by adding all these spaces to both tuples $\U$ and $\V$. The pair $(\U',\V')$ is a nontrivial solution.
\end{proof}

In this paper our objective is the description of all solutions of \cref{eq-main-counting-space} for $m = q + 1$, up to equivalence, and to study their properties. As we mentioned above, this task is reduced to the description of nontrivial solutions.

From \Cref{thm-nontrivial-solution-minimal-length} we see the importance of the coverings of a vector space by proper subspaces. In the case $m = q+ 1$ there is a description of all such coverings. 

\begin{lemma}[see \cite{d1}]\label{lemma-minimum-dense-covering-desc}
	Let $V$ be a nonzero space over $K$. Let $W_i \subset V$, for $i \in \myset{q+1}$, be proper subspaces and $V = \bigcup_{i = 1}^{q+1} W_i$. There exists subspace $S \subset V$ of codimension $2$ such that $\{W_1, \dots, W_m\}$ is the set of all hyperplanes in $V$ that contain $S$. The equality holds,
	\begin{equation}\label{eq-dense-covering}
	\sum_{i=1}^m \id_{W_i} = \id_V + q \id_S \;.
	\end{equation}
\end{lemma}
\begin{proof}
	The proof is based on the fact that a finite $K$-linear space can be covered by at least $q+1$ proper subspaces. Using the sieve theorem for the size of the covering, it is easy to prove that the only possible covering by the minimum number of the proper subspaces is the one presented in the statement. \Cref{eq-dense-covering} evidently follows.
\end{proof}

\section{Properties of minimal nontrivial solutions}
	In \cite{d1} we completely observed the solutions with different maximum dimensions of spaces in two tuples, $\max_{1 \leq i \leq m} \dim_K V_i \neq \max_{1 \leq i \leq m} \dim_K U_i$.
	
	\begin{proposition}[see \cite{d1}]\label{thm-type-a}
		Let $(\U,\V)$ be a nontrivial solution and  $$\max_{1 \leq i \leq q+1} \dim_K V_i > \max_{1 \leq i \leq q+1} \dim_K U_i\;.$$Then $(\U,\V)$ is equivalent to a solution of Type A.
	\end{proposition}
	\begin{proof}
		The proof is based on \Cref{lemma-minimum-dense-covering-desc}. Due to the different maximum dimensions, without a loss of generality, the space $V_1$ has dimension greater than one and is covered by the spaces $U_1,\dots, U_{q+1}$.
	\end{proof}

	\emph{From now in this section we suppose that $(\V, \U)$ is a nontrivial solution, $m = q+1$, $\max_{1 \leq i \leq m} \dim_K V_i = \max_{1 \leq i \leq m} \dim_K U_i = n > 1$ and the maximum is achieved on the spaces $V_1$ and $U_1$}.
	
	\begin{lemma}\label{lemma-properties}
	For all $i,j \in \myset{m}$ the following hold,
		\begin{itemize}
			\item[(a)] $U_i \neq V_j$,
			
			\item[(b)] $V_i \subseteq V_j$ implies $i = j$,
			
			\item[(c)] $\dim_K V_i \geq n-1$, $\dim_K U_j \geq n - 1$,
			%If $\dim_K U_j = n$, then $\dim_K V_i \cap U_j = n - 1$, for $i,j \in \myset{m}$,
			
			\item[(d)] if $\dim_K U_j > \dim_K V_i$, then $V_i \subset U_j$,
			
			\item[(e)] if $\dim_K U_j = n$, then there exists a subspace $S \subset U_j$ of dimension $n - 2$ such that $S \subset V_i$ for all $i \in \myset{m}$.
		\end{itemize}
	\end{lemma}
	\begin{proof}
		Assume that (a) does not hold.
		Without loss of generality, we can assume that $U_m = V_m$. Reduce \cref{eq-main-counting-space} to $\sum_{i = 1}^{m-1} \frac{1}{\card{V_i}} \id_{V_i} = \sum_{i=1}^{m-1} \frac{1}{\card{U_i}} \id_{U_i}$.
		The pair $\big( (V_1,\dots,V_{m-1}), (U_1,\dots,U_{m-1}) \big)$ is therefore a nontrivial solution of the new equation and, by \Cref{thm-nontrivial-solution-minimal-length}, $m - 1 \geq q +1$, which contradicts to the fact that $m = q+1$.

		Let $l \in \myset{m}$ be such that $\dim_K U_l = n$. Assume that for some $i,j \in \myset{m}$, $V_i \subseteq V_j$. Then $V_i \cap U_l \subseteq V_j \cap U_l$.
		Since $U_l$ has the largest dimension among all the spaces, according to (a), for all $t \in \myset{m}$, $V_t \cap U_l \neq U_l$.
		Also $U_l = \bigcup_{t = 1}^{m} V_t \cap U_l$. 
		From \Cref{lemma-minimum-dense-covering-desc}, since $m = q+1$, the spaces $V_t \cap U_l$, for $t \in \myset{m}$, are all different hyperplanes in $U_l$. Thus $V_i \cap U_l \subseteq V_j \cap U_l$ implies $i = j$ and $\dim_K V_i \cap U_l = n - 1$, for $i,j \in \myset{m}$. Therefore we have (b) and (c). If $\dim_K V_t = n - 1$, then $V_t = V_t \cap U_l \subset U_l$, which proves (d). Also, from \Cref{lemma-minimum-dense-covering-desc} there exists a space $S \subset U_l$ of dimension $n-2$, such that $S \subset V_t \cap U_l \subseteq V_t$, for all $t \in \myset{m}$. This proves (e).
	\end{proof}

	\begin{lemma}\label{lemma-common-subspace}
		There exists a space $S$ of dimension $n-2$ such that for all $i \in \myset{m}$, $S \subset U_i, V_i$.
	\end{lemma}
	\begin{proof}
		From \Cref{lemma-properties} (e), there exists a subspace $S \subset U_1$ of dimension $n-2$ such that for all $i \in \myset{m}$, $S \subset V_i$. Restrict both sides of  \cref{eq-main-counting-space} on the space $S$. In result we get,
		\begin{equation*}
		\sum_{i = 1}^{m} \frac{1}{\card{V_i}} \id_{S} = \frac{1}{\card{U_1}}\id_{S} + \sum_{i = 2}^m \frac{1}{\card{U_i}} \id_{U_i \cap S}\; \iff
		\end{equation*}
		\begin{equation}\label{t3e2}
		\iff \left(\sum_{i = 1}^{m} \frac{1}{\card{V_i}} - \frac{1}{\card{U_1}}\right) \id_{S} = \sum_{i = 2}^m \frac{1}{\card{U_i}} \id_{U_i \cap S}\;.
		\end{equation}
		Calculating \cref{eq-main-counting-space} in zero we get the equality, $\sum_{i = 1}^{m} \frac{1}{\card{V_i}} = \sum_{i = 1}^{m} \frac{1}{\card{U_i}}$.
		Thus the coefficient on the left side of \cref{t3e2} is positive.
		On the right side of \cref{t3e2} there are $m-1$ terms and therefore, by \Cref{lemma-minimum-size-of-space-equation}, there exists $i \in \{2,\dots, m\}$, without loss of generality assume $i = 2$, such that $S \subset U_2$. Continuing the procedure of elimination for all $i \in \{3,\dots,m\}$, we get $S \subset U_i$, for all $i \in \myset{m}$.
	\end{proof}

	\Cref{lemma-properties} (c) states that in a nontrivial solution the only possible dimensions of spaces are $n-1$ and $n$. Denote $X = \{ i \in \myset{m} \mid \dim_K V_i = n - 1 \}$ and $Y = \{ i \in \myset{m} \mid \dim_K U_i = n - 1 \}$.
	
\begin{lemma}\label{lemma-two-types}
	The cardinalities of $X$ and $Y$ are equal and they are not greater than one.
\end{lemma}
\begin{proof}
	Verify that $\card{X} = \card{Y}$. Calculate \cref{eq-main-counting-space} in the point $\{0\}$. Since all the spaces contain zero, we have $\sum_{i = 1}^{m} \frac{1}{\card{V_i}} = \sum_{i = 1}^{m} \frac{1}{\card{U_i}}$ or, the same, $\card{X} \frac{1}{q^{n-1}} + (m - \card{X}) \frac{1}{q^n} = \card{Y} \frac{1}{q^{n-1}} + (m - \card{Y}) \frac{1}{q^n}$. Hence $\card{X} = \card{Y}$.
		
	From \Cref{lemma-properties} (d), we have the inclusions,
	\begin{equation}\label{eq-xy-cap-cup}
	\bigcup_{i \in X} V_i \subseteq \bigcap_{i \notin Y} U_i \text{ and } \bigcup_{i \in Y} U_i \subseteq \bigcap_{i \notin X} V_i \;.
	\end{equation}
	Prove that $\card{X} > 1$ implies $\card{Y} \geq m - 1$. By the contradiction, assume that $\card{Y} < m - 1$. Inequality $\card{X}>1$ implies that there exist $i \neq j \in \myset{m}$ such that $\dim_K V_i = \dim_K V_j = n - 1$. By \Cref{lemma-properties} (b), $V_i \neq V_j$ and, by \Cref{lemma-properties} (e), $\dim_K V_i \cap V_j = n-2$. From \cref{eq-xy-cap-cup}, $V_i \cup V_j \subseteq \bigcap_{t \notin Y} U_t$ and, using the fact that $\card{ \myset{m} \setminus Y } \geq 2$, we have,
	\begin{equation*}
	2q^{n-1} - q^{n-2} = \card{V_i} + \card{V_j} - \card{V_i \cap V_j} = \card{V_i \cup V_j} \leq \card{\bigcap_{t \notin Y} U_t} \leq q^{n - 1}\;.
	\end{equation*}
	This inequality does not hold and hence, by contradiction, $\card{Y} \geq  m - 1$.
		
	The general assumption in the section is $\dim_KV_1 = n$, hence $1 \notin X$ and$\card{X} \leq m - 1$. Combining it with the result above, $\card{X} > 1$ implies $\card{Y} = \card{X} = m - 1$. 
	Prove that $\card{X} = m - 1$ is impossible. Assume that $\card{X} = m - 1$. This means that $\dim_K V_1 = \dim_K U_1 = n$ and $\dim_K U_i = \dim_K V_i = n -1$, for $i \in \{2,\dots,m\}$. Using \Cref{lemma-properties}, it is easy to see that $U_1 \cap V_i = V_i$ and $U_1 \cap U_i = S$, for $i \in \{2,\dots,m\}$, where $S \subset V_1$ is a space from \Cref{lemma-properties} (e) with $\dim_K S = n - 2$ and for all $i \in \myset{m}$, $S \subset U_i$.
	Calculate the restrictions of \cref{eq-main-counting-space} on $U_1$,
	\begin{equation}\label{teq}
		\frac{1}{q^n}\id_{U_1} + \frac{1}{q^{n-1}}\sum_{i=2}^m \id_{S} = \frac{1}{q^n} \id_{V_1 \cap U_1} + \frac{1}{q^{n-1}}\sum_{i = 2}^{m} \id_{V_i}\,\,,
	\end{equation}
	 \Cref{teq} implies $U_1 = (V_1 \cap U_1) \cup \bigcup_{i=2}^m V_i$, where the spaces $V_1 \cap U_1$, $V_2, \dots, V_m$ do not equal $U_1$. From \Cref{lemma-minimum-dense-covering-desc}, $\id_{U_1} + q \id_S = \sum_{i=2}^m \id_{V_i} + \id_{V_1 \cap U_1}$ on $U_1$. Substituting it to \cref{teq}, we get $\sum_{i=2}^m \id_{V_i} = q \id_S$ that is not true. By the contradiction, $\card{X} < m - 1$.
		
	In result, there are two possibilities, $\card{X} = \card{Y} = 0$ and $\card{X} = \card{Y} = 1$.
\end{proof}

So, there exist at most two possible dimension vectors for nontrivial solutions, with $\card{X} = \card{Y} = 1$ and $\card{X} = \card{Y} = 0$.
	
Let $S$ be the space with $\dim_K S = n - 2$ such that for all $i \in \myset{m}$, $S \subset U_i, V_i$. Let $Z_{ij}$ denote the space $U_i \cap V_j$ for $i, j \in \myset{m}$. Without loss of generality, assume that if $\card{X} = \card{Y} = 1$, then $\dim_K V_m = \dim_K U_m = n - 1$.
	
	\begin{lemma}\label{lemma-zij-properties}
		For all $i, j \in \myset{m}$ the following statements hold,
		\begin{itemize}
			\item [(a)]	if $i \neq j$, then $\dim_K Z_{ij} = n -1$,
			\item [(b)] if $\dim_K V_j = n$, then
			\begin{equation}\label{teq0}
			\sum_{i = 1}^m \id_{Z_{ij}} = \id_{V_j} + q\id_{S}\;,
			\end{equation}
			\item [(c)] for all $k,l \in \myset{m}$, such that $i \neq k$ or $j \neq l$, $Z_{ij} \cap Z_{kl} = S$.
		\end{itemize}
	\end{lemma}
	\begin{proof}
		Note that, for all $j \in \myset{m}$, we have $\bigcup_{i = 1}^m Z_{ij} = V_j$. Also, by \Cref{lemma-properties} (a) and (c), $\dim_K Z_{ij} = \dim_K U_i \cap V_j \leq n - 1$. If $\dim_K V_j = n$, then $Z_{ij}\subset V_j$ and, by \Cref{lemma-minimum-dense-covering-desc}, the spaces $Z_{ij}$ for $i \in \myset{m}$ form a covering of $V_j$ by hyperplanes that intersect in $S$. Therefore, for all $j \in \myset{m}$, such that $\dim_K V_j = n$, for all $i \in \myset{m}$, $\dim_K Z_{ij} = n -1$ and \cref{teq0} holds. If $j = m$, $\dim_K V_j = n - 1$ and $i \neq m$ then, by \Cref{lemma-properties} (d), $Z_{ij} = Z_{ji} = V_m$ and $\dim_K Z_{ij} = \dim_K Z_{ji} = \dim_K V_m = n - 1$.
		
		Consider again the equality $\bigcup_{i = 1}^m Z_{ij} = V_j$, for all $j \in \myset{m}$, and calculate $Z_{ij} \cap Z_{kl} = U_i \cap V_j \cap U_k \cap V_l = U_i \cap U_k \cap V_j \cap V_l$. By \Cref{lemma-properties}, $n -2 \leq \dim_K U_i \cap U_j, \dim_K V_j \cap V_l \leq n - 1$. If one of these two spaces $U_i \cap U_j$ or $V_j \cap V_l$ has dimension $n - 2$, then it is equal to $S$ and therefore $Z_{ij} \cap Z_{kl} = S$. Consider the case $\dim_K U_i \cap U_j = \dim_K V_k \cap V_l = n - 1$. This implies $\dim_K U_i = \dim_K U_j= \dim_K V_k =\dim_K V_l = n$. In this case, if $i \neq k$, $S \subseteq Z_{ij} \cap Z_{kl} \subseteq U_i \cap U_k \cap V_j = U_i \cap V_j \cap U_k \cap V_j = Z_{ij} \cap Z_{kj} = S$, since $Z_{ij}$ and $Z_{kj}$ are different hyperplanes in $V_j$. The same holds if $j \neq l$.
	\end{proof}

\section{Detailed description of minimal nontrivial solutions}
	
In this section we keep all the assumptions and notations from the previous section.

	\begin{proposition}\label{thm-factor-equation}
		Let $W$ be a $K$-space, $V_1, \dots, V_m, U_1, \dots, U_m\subseteq W$ be $K$-spaces that have a common subspace $S$. The equality $\sum_{i=1}^m \frac{1}{\card{V_i}} \id_{V_i} = \sum_{i=1}^m \frac{1}{\card{U_i}} \id_{U_i}$ of the functions on $W$ is equivalent to the equality 
		\begin{equation}\label{eq-factor-main-by-common}
		\sum_{i=1}^m \frac{1}{\card{V_i/S}} \id_{V_i/S} = \sum_{i=1}^m \frac{1}{\card{U_i/S}} \id_{U_i/S}
		\end{equation} of the functions on $W/S$.
	\end{proposition}
	\begin{proof}
		Let $S, V \subseteq W$ be spaces such that $S \subseteq V$. Let $\pi_S: W \rightarrow W/S$, $x \mapsto \bar{x}$ be a canonical projection, where $\bar{x} = x + S$. The equality $\id_V(x) = \id_{V/S}(\bar{x})$ holds. Indeed, $x \in V$ implies $\bar{x} \in V/S$, and conversely, $\bar{x} \in V/S$ implies that there exists $x' \in V$ such that $\bar{x} = \bar{x'}$, which is equivalent to the fact that $x - x' \in S$ and thus $x \in V$.
		
		Let $W_1, \dots, W_r \subseteq W$ be spaces with a common subspace $S$.
		Consider the function $F: W \rightarrow \mathbb{R}$, $a \mapsto \sum_{i=1}^r x_i \id_{W_i}(a)$, where $x_i \in \mathbb{R}$ for $i \in \myset{r}$. Let $\bar{F} = \sum_{i=1}^r x_i \id_{W_i/S}: W/S \rightarrow \mathbb{R}$. For each $x \in W$ the equality $F(x) = \bar{F}(\pi_S(x))$ holds, $\bar{F}(\pi_S(x)) = \sum_{i=1}^r x_i \id_{W_i/S}(\bar{x}) = \sum_{i=1}^r x_i \id_{W_i}(x) = F(x)$.
		Since $\pi_S$ is a projection, the identity $F \equiv 0$ on $W$ is equivalent to the identity $\bar{F} \equiv 0$ on $W/S$.
		
		Using the arguments above, the equation $\sum_{i=1}^m \frac{1}{\card{V_i}} \id_{V_i} = \sum_{i=1}^m \frac{1}{\card{U_i}} \id_{U_i}$ of the functions on $W$ is equivalent to the equation $\sum_{i=1}^m \frac{1}{\card{V_i}} \id_{V_i/S} = \sum_{i=1}^m \frac{1}{\card{U_i}} \id_{U_i/S}$ of the functions on $W/S$. Since $\card{V_i/S} = \card{V_i}/\card{S}$ and $\card{U_i/S} = \card{U_i}/\card{S}$, for all $i \in \myset{m}$, it is the same as $\card{S}\sum_{i=1}^m \frac{1}{\card{V_i/S}} \id_{V_i/S} = \card{S}\sum_{i=1}^m \frac{1}{\card{U_i/S}} \id_{U_i/S}$. The set $S$ contains zero, so $\card{S} > 0$ and we divide both sides of the equality by $S$, obtaining the necessary equality.
	\end{proof}
	
	Since the spaces in the nontrivial solution $(\U,\V)$ have a common subspace $S$ of dimension $n - 2$, we can factorize all the spaces by $S$ and describe nontrivial solutions in the case $S = \{0\}$. In such a way we can simplify sometimes the proof without loss of properties of nontrivial solutions. Therefore, we can assume that $n = 2$. We will use this assumption when we need it.
\comment{
We say that two solutions $(\V,\U)$ and $(\V',\U')$ are equivalent if $\V \sim \V'$ and $\U \sim \U'$ or $\V \sim \U'$ and $\U \sim \V'$. To simplify the formulation of the following proposition, let $\V/S$ denote the tuple $(V_1/S, \dots, V_m/S)$.

		\begin{lemma}\label{lemma-factor-solution-iff-original}
			Two solutions $(\U, \V)$ and $(\U', \V')$ are equivalent if and only if the corresponding two solutions of  \cref{eq-factor-main-by-common} $(\U/S, \V/S)$ and $(\U'/S, \V'/S)$ are equivalent.
		\end{lemma}
		\begin{proof}
			Obviously, if $S$ is a common subspace of $V_1, V_2$ in $W$, then $V_1 = V_2$ in $W$ if and only if $V_1/S = V_2/S$ in $W/S$.
		\end{proof}
}
		\begin{definition}
			Call a pair of tuples $(\U,\V)$ to be of Type B, if there exists a subspace $S \subset W$ and linearly independent vectors $a,b,c \in W$, where $S \cap \gen{a,b,c} = \{0\}$, such that $V_m = \gen{S, a}$, $V_i = \gen{S, b, \alpha_i a + c}$, $U_m = \gen{S, b}$ and $U_i = \gen{S, a, \alpha_i b + c}$, where $\alpha_i$ is the $i$th element in the field $K$, for $i \in \myset{q}$.
		\end{definition}

	In \Cref{fig-typeB} there is presented the inclusion diagram of spaces in a pair of Type B along with intersections.
	\begin{figure}[!ht]
		\centering
		\begin{tikzpicture}
		% First, locate each of the nodes and name them
		\node (topU1) at (-5,2) {$U_1$};
		\node (udotsl) at (-4,2) {$\dots$};
		\node (udots) at (-3,2) {$U_i$};
		\node (udotsr) at (-2,2) {$\dots$};
		\node (topUq) at (-1,2) {$U_{m-1}$};
		\node (midVm) at (-3,1) {$V_m$};
		\node (S) at (0,0) {$S$};
		\node (topV1) at (1,2) {$V_1$};
		\node (vdotsl) at (4,2) {$\dots$};
		\node (vdots) at (3,2) {$V_j$};
		\node (vdotsr) at (2,2) {$\dots$};
		\node (topVq) at (5,2) {$V_{m-1}$};
		\node (midUm) at (3,1) {$U_m$};
		\node (zij) at (0,1) {$Z_{ij}$};
		% Now draw the lines
		\draw  (topU1) -- (midVm) -- (S) -- (midVm) -- (topUq);
		\draw  (topV1) -- (midUm) -- (S) -- (midUm) -- (topVq);
		\draw (udots) -- (midVm);
		\draw (vdots) -- (midUm);
		\draw (udots) -- (zij) -- (S);
		\draw (vdots) -- (zij) -- (S);
		\end{tikzpicture}
		\caption{A pair of Type B}
		\label{fig-typeB}
	\end{figure}
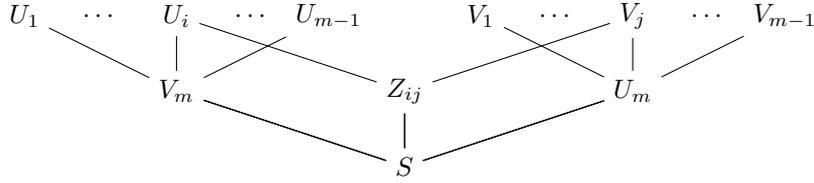

\begin{proposition}\label{thm-typeB-detailed-description}
	A pair $(\U,\V)$ of Type B is a nontrivial solution with $\card{X} = \card{Y} = 1$. If a pair $(\U,\V)$ is a nontrivial solution with $\card{X} = \card{Y} = 1$, then $(\U,\V)$ is equivalent to a solution of Type B.
\end{proposition}
\begin{proof}
To simplify both parts of the proof, according to \Cref{lemma-common-subspace} and \Cref{thm-factor-equation}, we can assume $S = \{0\}$ and $n = 2$.

Prove the first part. Calculate the intersection of the spaces, $Z_{ij} = U_i \cap V_j = \gen{a, \alpha_i b + c \rangle_K \cap \langle b, \alpha_j a + c}$, for $i,j \in \myset{q}$. After computations we get $Z_{ij} = \gen{\alpha_j a + \alpha_i  b + c}$. All the spaces $Z_{ij}$ for $i,j \in \myset{q}$, $V_m$ and $U_m$ are different. From \Cref{lemma-zij-properties} (b), $\sum_{i=1}^q \id_{Z_{ij}} + \id_{U_m} = \id_{V_j} + q\id_S$, for any $j \in \myset{q}$. Note that $V_i \cap V_j = \gen{b} = U_m$, for $i \neq j\in \myset{q}$.

For all $j \in \myset{q}$ calculate the projection of both sides of \cref{eq-main-counting-space}, multiplied from both sides by $q^n$, on $V_j$, where $j \in \myset{q}$,
\begin{equation}\label{teq1}
	\left.\left(\sum_{i=1}^{q} \id_{V_i} + q \id_{V_m}\right) \right\vert_{V_j} = \sum_{i=1}^{q} \id_{V_i \cap V_j} + q \id_{V_m \cap V_j} =
	\id_{V_j} + (q-1) \id_{U_m} + q\id_{V_m \cap V_j}\;,
\end{equation}
\begin{equation}\label{teq2}
	\left.\left(\sum_{i=1}^{q} \id_{U_i} + q\id_{U_m}\right) \right\vert_{V_j} = \sum_{i=1}^{q} \id_{U_i \cap V_j} + q \id_{U_m \cap V_j} =
	\sum_{i=1}^{q} \id_{Z_{ij}} + q \id_{U_m}\;.
\end{equation}

Considering the fact that $V_m \cap V_j = S$, the projection of the left and the right side of \cref{eq-main-counting-space} are equal and therefore the pair $(\U,\V)$ of the Type B is a nontrivial solution. It is easy to see that $\card{X} = \card{Y} = 1$, the corresponding spaces of dimension $n-1$ are $V_m$ and $U_m$.

Prove the second part. Let $(\U,\V)$ be a nontrivial solution with $\card{X} = \card{Y} = 1$. At first, note some properties of the spaces in $\V$ and $\U$. From \Cref{lemma-properties} (d), $U_m \subset V_i$ and $V_m \subset U_i$ for all $i \in \myset{m}$. As a result, using \Cref{lemma-properties} (b), $V_i \cap V_j = U_m$ and $U_i \cap U_j = V_m$ for all $i \in \myset{m}$. Also, $V_m \cap V_i = U_m \cap U_i = S$ for $i \in \myset{q}$. From \Cref{lemma-zij-properties} (c), all the spaces $V_m$, $U_m$ and $Z_{ij}$, $i \in \myset{q}$ are different.

Let $a,b,c \in W$ be three vectors, such that $V_m = \langle a \rangle_K$, $U_m = \langle b \rangle_K$, $Z_{11} = \langle c \rangle_K$. From the properties that we mentioned above, the spaces $V_m$, $U_m$ and $Z_{11}$ are all different so the vectors $a,b,c$ are pairwise linearly independent. Obviously, $V_1 =\langle b,c \rangle_K$ and $U_1 = \langle a, c\rangle_K$. \Cref{lemma-properties} (a) states that $V_1 \neq U_1$ and thus all three vectors $a,b,c \in W$ are linearly independent.

The plane $U_1$ is covered by $m$ different lines $V_m, Z_{11}, \dots, Z_{1q}$. Let $v_2, \dots, v_q$ be such that $Z_{1i} = \gen{v_i}$  for $i \in \{2, \dots, q\}$. In the same way, the plane $V_1$ is covered by the lines $U_m, Z_{11}, \dots, Z_{q1}$. Let $Z_{i1} = \langle u_i \rangle_K$, for some $u_i \in W$, where $i \in \{2, \dots, q\}$.

In \Cref{tab-type-b} there are illustrated the intersections of the spaces $V_1, \dots, V_m$, $U_1, \dots, U_m$. Note that if we calculate the union of all spaces in a row of the table, we get the space that corresponds to the row. The same is for the columns of the tables.
To satisfy this requirement, for all $i \in \{ 2, \dots, q \}$, $v_i \in \langle a, c \rangle_K$, $u_i \in \langle b, c \rangle_K$,  $V_i = \langle b, v_i \rangle_K$ and $U_i = \langle a, u_i \rangle_K$. Hence $v_i = \alpha_i a + \beta_i c$, $u_i = \gamma_i b + \delta_i c$ for some $\alpha_i, \beta_i, \gamma_i, \delta_i \in K$, $i \in \{2,\dots,q\}$. Then $Z_{ij} = U_i \cap V_j = \langle a, u_i \rangle_K \cap \langle b, v_j \rangle_K= \langle a, \gamma_i b + \delta_i c \rangle_K \cap \langle b, \alpha_j a + \beta_j c \rangle_K$. After computations we get $Z_{ij} = \langle \delta_i \alpha_j a + \gamma_i \beta_j  b + \delta_i \beta_j c \rangle_K$ or the same, since $b_j \neq 0$ and $\delta_i \neq 0$, $Z_{ij} = \langle \frac{\alpha_j}{\beta_j} a + \frac{\gamma_i}{\delta_i}  b + c \rangle_K$. As we have mentioned before, all the spaces $Z_{ij}$ for $i,j \in \myset{q}$ should be different, thus the values $\frac{\alpha_i}{\beta_i}$ and  $\frac{\gamma_i}{\delta_i}$ should both run through $K$ while $i$ runs through $\myset{q}$. As a result, we showed that the pair $(\U,\V)$ is, up to an order of spaces in tuples, exactly of the Type B.
\end{proof}

			\begin{table}[!ht]
				\centering
				\begin{tabular}{c|c|c|c|c|c|c|}
					$\cap$ & $V_m$ & $V_1$ & \dots & $V_i$ & \dots & $V_q$ \\
					\hline
					$U_m$ & $\{0\}$ & $\gen{b}$ & \dots  & $\gen{b}$ & \dots & $\gen{b}$ \\
					\hline
					$U_1$ & $\gen{a}$ & $\gen{c}$ & \dots & $\gen{v_i}$ & \dots & $\gen{v_q}$ \\
					\hline
					$\vdots$ & $\vdots$ & $\vdots$ & $\ddots$ & $\vdots$ & $\vdots$ & $\vdots$ \\
					\hline
					$U_j$ & $\gen{a}$ & $\gen{u_j}$ & $\dots$ & $Z_{ji}$ & $\dots$ & $Z_{jq}$ \\
					\hline
					$\vdots$ & $\vdots$  & $\vdots$ & $\vdots$ &  $\vdots$& $\ddots$ & $\vdots$ \\
					\hline
					$U_q$ & $\gen{a}$ & $\gen{u_q}$ & $\dots$ & $Z_{qi}$ & $\dots$ & $Z_{qq}$ \\
					\hline
				\end{tabular}
				\caption{Intersection table for a solutions of Type B}
				\label{tab-type-b}
			\end{table}

\begin{definition}
	Say that a pair $(\U,\V)$ if of Type C, if there exist a subspace $S \subset W$ and linearly independent vectors $a,b,c,d \in W$, where $S \cap \gen{a,b,c,d} = \{0\}$, such that $V_i  = \gen{S, \alpha_i a + \beta_i b, \alpha_i c + \beta_i d}$ and $U_i = \gen{S, \alpha_i a + \beta_i c, \alpha_i b + \beta_i d}$, where $[\alpha_i : \beta_i]$ is the $i$th element in $\mathbb{P}_1(K)$, for $i \in \myset{m}$.
\end{definition}
In \Cref{fig-typeC} there is presented the inclusion diagram of spaces in a pair of Type C along with intersections.

\begin{figure}[!ht]
	\centering
	\begin{tikzpicture}
	% First, locate each of the nodes and name them
	\node (topU1) at (-5,2) {$U_1$};
	\node (udotsl) at (-4,2) {$\dots$};
	\node (udots) at (-3,2) {$U_i$};
	\node (udotsr) at (-2,2) {$\dots$};
	\node (topUq) at (-1,2) {$U_{m}$};
	\node (S) at (0,0) {$S$};
	\node (topV1) at (1,2) {$V_1$};
	\node (vdotsl) at (4,2) {$\dots$};
	\node (vdots) at (3,2) {$V_j$};
	\node (vdotsr) at (2,2) {$\dots$};
	\node (topVq) at (5,2) {$V_{m}$};
	\node (zij) at (0,1) {$Z_{ij}$};
	% Now draw the lines
	\draw  (topU1) -- (S) -- (udots) -- (S) -- (topUq);
	\draw  (topV1) -- (S) -- (vdots) -- (S) -- (topVq);
	\draw (udots) -- (zij) -- (S);
	\draw (vdots) -- (zij) -- (S);
	\end{tikzpicture}
	\caption{
		A pair of Type C}
	\label{fig-typeC}
\end{figure}
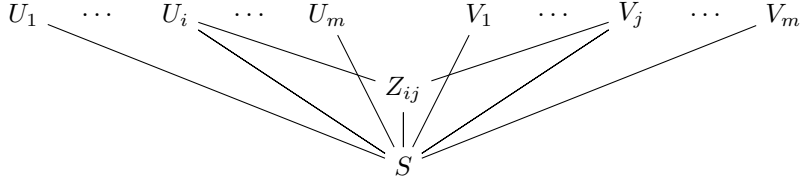

\begin{proposition}\label{thm-typeC-detailed-description}
	A pair $(\U,\V)$ of Type C is a nontrivial solution with $\card{X} = \card{Y} = 0$. If a pair $(\U,\V)$ is a nontrivial solution with $\card{X} = \card{Y} = 0$, then $(\U,\V)$ is equivalent to a solution of Type C.			
\end{proposition}
\begin{proof}
To simplify both parts of the proof, according to \Cref{lemma-common-subspace} and \Cref{thm-factor-equation}, we can assume $S = \{0\}$ and $n = 2$.

Prove the first part. Calculate the intersection of the spaces, $Z_{ij} = U_i \cap V_j = \gen{\alpha_i a + \beta_i c, \alpha_i b + \beta_i d} \cap \gen{\alpha_j a + \beta_j b, \alpha_j c + \beta_j d}$, for $i,j \in \myset{m}$. After computations we get $Z_{ij} = \gen{\alpha_i \alpha_j a + \beta_i \alpha_j b + \alpha_i \beta_j c + \beta_i \beta_j d}$. All the spaces $Z_{ij}$ for $i,j \in \myset{m}$ are different. From \Cref{lemma-minimum-dense-covering-desc}, $\sum_{i=1}^m \id_{Z_{ij}} = \id_{V_j} + q\id_S$, for any $j \in \myset{m}$. Note that $V_i \cap V_j = \{0\}$, for $i \neq j \in \myset{m}$.

We calculate for $j \in \myset{m}$ the restriction of both sides of \cref{eq-main-counting-space}, multiplied by $q^n$, on the space $V_j$,
\begin{equation}\label{t2eq1}
	\left.\left(\sum_{i=1}^m \id_{V_i} \right)\right\vert_{V_j} =
	\id_{V_j} + \sum_{i \neq j} \id_{V_i \cap V_j}\,\,,
\end{equation}
\begin{equation}\label{t2eq2}
	\left.\left(\sum_{i=1}^m \id_{U_i}\right) \right\vert_{V_j} = \sum_{i=1}^m \id_{U_i \cap V_j} =
	\sum_{i=1}^m \id_{Z_{ij}}\,\,.
\end{equation}

Obviously, the pair $(\U,\V)$ of Type C satisfy these equations for any $j \in \myset{m}$, and therefore is a nontrivial solution with $\card{X}= \card{Y} = 0$.

Prove the second part. Let $(\U,\V)$ be a nontrivial solution with $\card{X} = \card{Y} = 0$. Using \cref{teq0}, and the fact that the right sides of \cref{t2eq1} and \cref{t2eq2} are equal, for the fixed $j$, $\sum_{i \neq j} \id_{V_i \cap V_j} = q \id_{S}$. From \Cref{lemma-minimum-dense-covering-desc} for all $i \neq j \in \myset{m}$, $V_i \cap V_j = S$. In the same way for all $i \neq j \in \myset{m}$, $U_i \cap U_j = S$.
From \Cref{lemma-zij-properties} for all $i,j \in \myset{m}$, $\dim_K Z_{ij} = n-1$ and the spaces $Z_{ij}$, $i,j \in \myset{m}$ are all different.

Let $a, b, c, d \in W$ be such that $Z_{11} = \langle a \rangle_K$, $Z_{12} = \langle b \rangle_K$, $Z_{21} = \langle c \rangle_K$ and $Z_{22} = \langle d \rangle_K$. We deduce that $V_1 = \langle a, c \rangle_K$, $V_2 = \langle b, d \rangle_K$, $U_1 = \langle a, b \rangle_K$ and $U_2 = \langle c,d \rangle_K$. The intersection $V_1 \cap V_2 = \{0\}$ and $\dim_K V_1 = \dim_K V_2 = 2$ that implies the linearly independence of the vectors $a,b,c,d$.

Consider the \Cref{tab-type-c} where in the cells there are the one-dimensional spaces $Z_{ij} = U_i \cap V_j$. The union of the lines in each row gives the space that represents the row and the union of the lines in each column gives the space that represents the column.

Let $\alpha_i, \beta_i, \alpha'_i, \beta'_i, \gamma_i, \delta_i, \gamma'_i, \delta'_i \in K$ be such that $Z_{1i} = \langle \alpha_i a + \beta_i b \rangle_K$, $Z_{i1} = \langle \gamma_i a + \delta_i c \rangle_K$, $Z_{i2} = \langle \gamma'_i b + \delta'_i d \rangle_K$ and $Z_{2i} = \langle \alpha'_i c + \beta'_i d \rangle_K$ for $i \in \myset{m}$. 
With the defined coefficients we get $V_i = \langle \alpha_i a + \beta_i b, \alpha'_i c + \beta'_i d \rangle_K$ and $U_i = \langle \gamma_i a + \delta_i c, \gamma'_i b + \delta'_i d \rangle_K$. Since all the spaces $Z_{ji}$, $i,j \in \myset{m}$ should be different, the equalities 
$\alpha_i'\beta_i = \beta_i' \alpha_i$, $\gamma'_i \delta_i = \delta_i' \gamma_i$ hold for all $i \in \myset{m}$ and the intersection space is $Z_{ji} = \gen{ \alpha_i\gamma_j a + \gamma_j \beta_i b  + \alpha_i \delta_j c + \delta_j \beta_i d}$. It is easy to verify that $(\U,\V)$ is of Type C.
\end{proof}
	
			\begin{table}[!ht]
				\centering
				\begin{tabular}{c|c|c|c|c|c|c|}
					$\cap$& $V_1$ & $V_2$ & \dots & $V_i$ & \dots & $V_m$ \\
					\hline
					$U_1$ & $\gen{a}$ & $\gen{b}$ & \dots  & & \dots &  \\
					\hline
					$U_2$ & $\gen{c}$ & $\gen{d}$ & \dots &  & \dots &  \\
					\hline
					$\vdots$ & $\vdots$ & $\vdots$ & $\ddots$ & $\vdots$ & $\vdots$ & $\vdots$ \\
					\hline
					$U_j$ &  &  & $\dots$ & $Z_{ji}$ & $\dots$ & $Z_{jm}$ \\
					\hline
					$\vdots$ & $\vdots$  & $\vdots$ & $\vdots$ &  $\vdots$& $\ddots$ & $\vdots$ \\
					\hline
					$U_m$ & &  & $\dots$ & $Z_{mi}$ & $\dots$ & $Z_{mm}$ \\
					\hline
				\end{tabular}
				\caption{Intersection table for the solutions of Type C}
				\label{tab-type-c}
			\end{table}

\begin{theorem}\label{thm-main-abc}
	Let $(\U,\V)$ be a nontrivial solution of \cref{eq-main-counting-space} with $m = \card{K} + 1$. Up to an equivalence, the pair $(\U,\V)$ is of one of the following types: Type A, Type B or Type C.
\end{theorem}
\begin{proof}
	If $\max_{1 \leq i \leq q+1} \dim_K V_i \neq \max_{1 \leq i \leq q+1} \dim_K U_i$, by \Cref{thm-type-a}, $(\U,\V)$ is equivalent to a pair of Type A. If $\max_{1 \leq i \leq q+1} \dim_K V_i = \max_{1 \leq i \leq q+1} \dim_K U_i$, by \Cref{lemma-two-types}, either $\card{X} = \card{Y} = 1$ or $\card{X} = \card{Y} = 0$. In the first case, using \Cref{thm-typeB-detailed-description}, $(\U,\V)$ is equivalent to a pair of Type B. In the first case, using \Cref{thm-typeC-detailed-description}, $(\U,\V)$ is equivalent to a pair of Type C.
\end{proof}

After the full description and classification of all the nontrivial minimal solutions have been made, we can prove some interesting facts on their properties. 
		
\begin{proposition}\label{thm-number-solutions-unique-nontrivial}
	Let $m = q+1$. For any tuple of spaces $\V$ there exists at most one tuple of spaces $\U$, up to an equivalence, such that $(\U,\V)$ is a nontrivial solution.
\end{proposition}
		\begin{proof}
The statement is obvious for the solution of Type A.
			By \Cref{lemma-common-subspace} and \Cref{thm-factor-equation}, we can assume $S = \{0\}$.
			Consider the solution $(\U,\V)$ of Type B. Having the tuple $\U$ we can uniquely, up to equivalence, recover the tuple $\V$. Really, at first, we recover the space $V_m = \gen{a}$ as the intersection of any two two-dimensional spaces and we already have $U_m = \gen{b}$. Assume that there exists another solution $(\U, \V')$, where $\V' = (V_1',\dots, V_m')$. But then $V'_i = \gen{b, x_i}$, $x_i \in U_1$ for $i \in \myset{q}$, where $\{x_i\}_{i \in \myset{m}}$ is a set of all lines in $U_1$. Therefore $\V \sim \V'$.
			
			Consider the solution $(\U,\V)$ of Type C. Using the notations from \Cref{tab-type-c}, the vector $a+c$ is in $V_1$ and in $U_i$, for some $i \in \myset{m}$. Let $(\U,\V')$ be another solution. For this solution let $i,j \in \myset{m}$ be such that $a \in V_i \cap U_1$ and $b \in V'_j \cap U_1$. Also, let $c',d' \in W$ be such that $\gen{c'} = V'_i \cap U_2$ and $\gen{d'} = V'_j \cap U_2$. Then $V'_i = \gen{a,c'}$, $V'_j = \gen{b,d'}$. The vector $a + c$ should be presented in some intersection with $V'_k$, for $k \in \myset{m}$. This is only possible if $\gen{c} = \gen{c'}$. In the same way, observing the vector $b +d$, we deduce $\gen{d'} = \gen{d}$ and thus $\V \sim \V'$.
		\end{proof}
		
		\begin{proposition}\label{thm-sum-of-two-spaces}
			Let $(\U,\V)$ be a nontrivial solution with $m = q+1$. For any $i\neq j \in \myset{m}$, for any $k \in \myset{m}$, $V_k,U_k \subseteq V_i + V_j$. For any $i\neq j \in \myset{m}$, $\dim_K V_i + V_j \leq 2 + \max_{i \in \myset{m}} \dim_K V_i$.
		\end{proposition}
		\begin{proof}
			The description of all the nontrivial solutions for the codes of the length $m = q + 1$ is given in \Cref{thm-main-abc}. Let $n = \max_{i \in \myset{m}} \dim_K V_i$ and fix $i,j \in \myset{m}$, $i\neq j$.
			If the solution is of Type B, then the space $V_i+V_j$ is of dimension $\dim_K V_1 + 1 = n + 1$ and contains all the spaces $V_1, \dots, V_m$. If the solution is of Type C, then the space $V_i+ V_j$ has the dimension $\dim_K V_1 +2 = n + 2$ and contains all the spaces $V_1,\dots, V_m$. 
			Regarding the solution of Type A, depending on which tuple of spaces we observe, the space $V_i + V_j$ contains all the spaces from the tuple and has the dimension $n$ or $n+1$.
			Combining these three cases, all the spaces $V_1, \dots, V_m$ are in the space $V_i + V_j$ and therefore the spaces $U_1, \dots, U_m$ are all in $V_i + V_j$. Also, $\dim_K V_i + V_j \leq n + 2$. 
		\end{proof}
		
		\comment{
			Let $(\cdot, \cdot) : W \times W \rightarrow K$ be a $K$-bilinear non-degenerate form. For a $K$-linear space $V \subseteq W$ let $V^{\perp} = \{w \in W \mid \forall v \in V: (w,v) = 0 \}$ be the orthogonal space. Let $f: W \rightarrow \mathbb{C}$ be a function. Fourier transformation $\mathcal{F}$ of $f$ is defined as $\mathcal{F}(f)(s) = \sum_{w \in W} f(w) \chi_s(w)$, for $s \in W$, where $\chi_s(w) = \xi^{\tr_{K/\mathbb{F}_p}(sw)}$, $s,w \in K$. The inverse Fourier transformation is given by the following formula: $\mathcal{F}^{-1}(f)(s) = \frac{1}{\card{W}} \sum_{w \in W} f(w) \overline{\chi_s(w)}$.
			
			\begin{proposition}
				The $K$-linear map $f: C \rightarrow L^m$ is an isometry if and only if
				\begin{equation}\label{eq-main-dual}
				\sum_{i = 1}^{m} \id_{V_i^{\perp}} = \sum_{i = 1}^{m} \id_{U_i^{\perp}}\;.
				\end{equation}
			\end{proposition}
			\begin{proof}
				Evidently, $\chi_w(v) = 1$ if and only if $(w,v) = 0$, for $w,v \in W$. Then $\mathcal{F}(\id_V)(u) = \sum_{w \in W} \id_V(w) \chi_w(u) = \sum_{w \in V} \chi_w(u) = \card{V} \id_{V^{\perp}}(u)$. Thus $\mathcal{F}\left( \sum_{i= 1}^m \frac{1}{\card{V_i}} \id_{V_i} \right) = \sum_{i= 1}^m \id_{V_i^{\perp}}$. This means that any solution of \cref{eq-main-counting-space} is a solutions of the \cref{eq-main-dual}. Taking the inverse Fourier transform on both parts of \cref{eq-main-dual}, we get the statement in the other direction. In such a way, using \Cref{thm-isometry-criterium}, we prove the proposition.
			\end{proof}
		}

\end{document}